\documentclass[12pt,reqno]{amsart}
\usepackage[top=3cm,bottom=3cm,left=3cm,right=3cm,a4paper]{geometry}
\usepackage[T1]{fontenc}
\usepackage{tgtermes}
\usepackage[cal=cm]{mathalfa}
\usepackage{helvet}
\usepackage[colorlinks, linkcolor=blue, citecolor=blue, urlcolor=blue]{hyperref}
\usepackage{enumitem}
\usepackage{tikz}
\usepackage[labelfont={}]{subcaption}
\usepackage{mathtools}
\usepackage{amssymb}
\usepackage{color}
\usepackage{hyperref}
\makeatletter
\@namedef{subjclassname@2010}{%
  \textup{2010} Mathematics Subject Classification}
\makeatother

\newtheorem{theorem}{Theorem}[section]
\newtheorem{corollary}[theorem]{Corollary}
\newtheorem{fact}[theorem]{Fact}
\newtheorem{lemma}[theorem]{Lemma}
\newtheorem{proposition}[theorem]{Proposition}
\newtheorem{question}[theorem]{Question}
\newtheorem{claim}[theorem]{Claim}

\theoremstyle{definition}
\newtheorem{definition}[theorem]{Definition}
\newtheorem{remark}[theorem]{Remark}

\DeclareMathOperator{\cof}{cof}
\DeclareMathOperator{\dom}{dom}

\DeclareMathOperator{\ult}{\mathrm{Ult}}

\DeclareMathOperator{\sch}{\mathrm{SCH}}

\DeclareMathOperator{\Lev}{\mathrm{Lev}}

\DeclareMathOperator{\hod}{\mathrm{HOD}}
\DeclareMathOperator{\add}{\mathrm{Add}}
\DeclareMathOperator{\hei}{\mathrm{ht}}
\DeclareMathOperator{\gch}{\mathrm{GCH}}
\DeclareMathOperator{\zfc}{\mathrm{ZFC}}
\def\br{\blacktriangleright}
\begin{document}
\bibliographystyle{alpha}
\baselineskip=17pt

\title[How far is almost strong compactness from strong compactness]{How far is almost strong compactness from strong compactness}

\author[Zhixing You]{Zhixing You}
\address{Department of Mathematics, Bar-Ilan University, Ramat-Gan 5290002, Israel}
\email{zhixingy121@gmail.com}

\author[Jiachen Yuan]{Jiachen Yuan}
\address{School of Mathematics, University of Leeds, Leeds LS2 9JT, UK}
\email{J.Yuan@leeds.ac.uk}

\date{}
\thanks{
The first author was supported by the European Research Council (grant agreement ERC-2018-StG 802756).
The second author was supported by a UKRI Future Leaders Fellowship [MR/T021705/2].
For the purpose of open access, the author(s) has applied a Creative Commons
Attribution (CC BY) licence to any Author Accepted Manuscript version arising.}
\begin{abstract}
Bagaria and Magidor introduced the notion of almost strong compactness,
which is very close to the notion of strong compactness.
Boney and Brooke-Taylor asked whether the least almost strongly compact cardinal is strongly compact.
Goldberg gives a positive answer in the case $\sch$ holds from below and the least almost strongly compact cardinal has uncountable cofinality.
In this paper, we give a negative answer for the general case.
Our result also gives an affirmative answer to a question of Bagaria and Magidor.
\end{abstract}
\maketitle
\section{introduction}
The notions of $\delta$-strong compactness and almost strong compactness (see Definition \ref{dscc})
were introduced by Bagaria and Magidor in \cite{BM2014,BMO2014}.
They are weak versions of strong compactness,
and characterize many natural compactness properties of interest in different areas.
See \cite{BM2014,BMO2014,U2020} for details.

Like strong compactness, $\delta$-strong compactness can be characterized
in terms of compactness properties of infinitary languages, elementary embeddings, ultrafilters, etc..
In addition, many interesting properties following from strong compactness
are a consequence of $\delta$-strong compactness.
For example, if $\kappa$ is the least $\delta$-strongly compact cardinal for some uncountable
cardinal $\delta$, then the \emph{Singular Cardinal Hypothesis} ($\sch$) holds above $\kappa$;
and for every regular cardinal $\lambda \geq \kappa$,
stationary reflection holds for every stationary subset of $\mathrm{S}^{\lambda}_{<\delta}=\{\alpha<\lambda \ | \ \cof(\alpha)<\delta \}$.
In addition, Goldberg in \cite[Corollary 2.9]{Gol21} proved that Woodin's $\mathrm{HOD}$ Dichotomy\footnote{
Suppose $\kappa$ is $\omega_1$-strongly compact.
Then either all sufficiently large regular cardinals are measurable in $\hod$
or every singular cardinal
$\lambda$ greater than $\kappa$ is singular in $\hod$ and $\lambda^{+N}=\lambda^+$.}
is a consequence of $\omega_1$-strong compactness.

$\delta$-strong compactness, almost strong compactness, and strong compactness are close compactness principles.
Magidor in \cite{Mag76} proved that
consistently the least strongly compact cardinal, say $\kappa$, is the least measurable cardinal. In this case,
$\kappa$ is also the least $\delta$-strongly compact cardinal for every uncountable cardinal $\delta<\kappa$ and the least almost strongly compact cardinal.
In addition, Goldberg in \cite[Proposition 8.3.7]{Gol20} proved that under the assumption of the \emph{Ultrapower Axiom},
which is expected to hold in all canonical inner models, for any uncountable cardinal $\delta$,
the least $\delta$-strongly compact cardinal is strongly compact.

On the other hand, Bagaria and Magidor in \cite{BMO2014}
turned a supercompact cardinal $\kappa$ into
the least $\delta$-strongly compact cardinal by using a suitable Radin forcing of length some measurable cardinal $\delta<\kappa$.
Thus in the generic extension, $\kappa$ is singular, which implies it is not strongly compact.
This separates $\delta$-strong compactness from strong compactness.

The more subtle case is between almost strong compactness and strong compactness.
Obviously if $\kappa$ is a strongly compact cardinal, the successor of a strongly compact cardinal,
or a limit of almost strongly compact cardinals, then it is almost strongly compact.
For the other direction,
Menas essentially proved that
if an almost strongly compact cardinal is measurable, then it is strongly compact (see \cite[Theorem 22.19]{Kana}).
Recently, Goldberg in \cite[Theorem 5.7]{G2020} proved that
assuming the $\sch$ holds, every almost strongly compact cardinal $\kappa$ of uncountable cofinality is trivial, i.e.,
one of the three cases mentioned above.
In particular, noting that $\sch$ holds above the least almost strongly compact cardinal, Goldberg \cite[Theorem 5.8]{G2020} proved that for every ordinal $\alpha>0$, if the $(\alpha + 1)$-st almost strongly compact limit cardinal has uncountable cofinality, then it is strongly compact.

But the following question posed by
Boney and Brooke-Taylor remained open:
\begin{question}[\cite{G2020}\label{q2}]
Is the least almost strongly compact cardinal necessarily strongly compact?
\end{question}

We will show Theorem 5.7 and Theorem 5.8 in \cite{G2020} may be no longer true when the
cofinality assumption is dropped in Section \ref{sec4}.
The point is that Fodor's lemma does not hold for the least infinite cardinal $\omega$.
This answers Question \ref{q2} in the negative. 

To achieve this,
we need a positive answer to the following question of Bagaria and Magidor:
\begin{question}[Bagaria, Magidor]\label{QBM}
Is there a class (possibly proper) $\mathcal{K}$ with $|\mathcal{K}| \geq 2$, and a $\delta_{\kappa}<\kappa$ for every $\kappa \in \mathcal{K}$, so that $\kappa$ is the least exactly $\delta_{\kappa}$-strongly compact cardinal for every $\kappa \in \mathcal{K}$?
\end{question}
For example, if we can give an affirmative answer to the above question when $\mathcal{K}$ has order type $\omega$,
and $\delta_{\kappa}<\kappa$ is a measurable cardinal above $\max(\mathcal{K} \cap \kappa)$ for every $\kappa \in \mathcal{K}$,
then $\sup(\mathcal{K})$ may be the least almost strongly compact cardinal.

We give a positive answer for Question \ref{QBM} in Section \ref{sec4}. To achieve this,
we recall a new construction of Gitik in \cite[Theorem 3.1]{Git2020}.
He developed Kunen's basic idea  of a construction of a model with a $\kappa$-saturated ideal over $\kappa$ (see \cite{Kun78}).
After some preparation, he added a $\delta$-ascent $\kappa$-Suslin tree at a supercompact cardinal $\kappa$.
This turned $\kappa$ into the least $\delta$-strongly compact cardinal.
Further analysis shows that $\kappa$ is not $\delta^+$-strongly compact,
which means that we may control the compactness of $\kappa$.
In addition, the forcing for adding a $\delta$-ascent $\kappa$-Suslin tree has nice closure properties, i.e., ${<}\kappa$-strategically closed and ${<}\delta$-directed closed.
So we may apply this method and the well-known method of producing class many non-supercompact strongly compact cardinals simultaneously,
to get hierarchies of $\delta$-strongly compact cardinals for different $\delta$ simultaneously.
Thus we may obtain an affirmative answer to Question \ref{QBM}.
\subsection*{The structure of the paper}
In this paper, Section \ref{sec2} covers some technical preliminary information and basic definitions. In Section \ref{sec3},
we give some variants of Gitik's construction.
In Section \ref{sec4}, building on the results of the previous section,
we construct hierarchies of $\delta$-strongly compact cardinals for different $\delta$ simultaneously
and provide nontrivial examples of almost strongly compact cardinals, in particular answering Question \ref{q2} and Question \ref{QBM}.
\section{preliminaries}\label{sec2}
\subsection{Large cardinals} \label{scc}
We assume the reader is familiar with the large cardinal notions of Mahloness, measurability, strongness, strong compactness, and supercompactness (see \cite{J2003} or \cite{Kana} for details).

We first review the definitions and basic properties of $\delta$-strongly compact cardinals and almost strongly compact cardinals.
\begin{definition}(\cite{BM2014,BMO2014})\label{dscc}
Suppose $\kappa \geq \delta$ are two uncountable cardinals,
\begin{enumerate}
  \item For any $\theta \geq \kappa$, $\kappa$ is \emph{$(\delta,\theta)$-strongly compact} if there is an elementary embedding $j:V \rightarrow M$ with $M$ transitive such that $\mathrm{crit}(j) \geq \delta$, and there is a $D \in M$ such that $j''\theta \subseteq D$ and $M \models |D|<j(\kappa)$.
  \item $\kappa$ is \emph{$\delta$-strongly compact} if $\kappa$ is $(\delta,\theta)$-strongly compact for any $\theta \geq \kappa$;
  $\kappa$ is \emph{exactly $\delta$-strongly compact} if $\kappa$ is $\delta$-strongly compact but not $\delta^+$-strongly compact.
  \item $\kappa$ is \emph{almost strongly compact} if $\kappa$ is $\delta$-strongly compact for every uncountable cardinal $\delta<\kappa$.
\end{enumerate}
\end{definition}
By the definition above, it is easy to see that $\kappa$ is $\kappa$-strongly compact if and only if $\kappa$ is strongly compact, and the former implies the latter for these sentences: $\kappa$ is strongly compact, $\kappa$ is almost strongly compact, $\kappa$ is $\delta'$-strong compactness with $\delta<\delta'<\kappa$, and $\kappa$ is $\delta$-strongly compact.

Usuba characterized $\delta$-strong compactness in terms of $\delta$-complete uniform ultrafilters, which generalized Ketonen's result.
\begin{theorem}[\cite{U2020}]\label{tKU}
Suppose $\kappa \geq \delta$ are two uncountable cardinals. Then $\kappa$
is $\delta$-strongly compact if and only if for every regular $\lambda \geq \kappa$, there is a
$\delta$-complete uniform ultrafilter over $\lambda$, i.e.,
there ia a $\delta$-complete ultrafilter $U$ over $\lambda$ such that every $A \in U$ has cardinality $\lambda$.
\end{theorem}

Next, we list two useful lemmas (see \cite[Lemma 2.1, Lemma 2.4]{AC2001} for details).
\begin{lemma}\label{lem10}
Let $\kappa$ be $2^\kappa$-supercompact and strong.
Assume $j : V \rightarrow M$
is a $2^{\kappa}$-supercompact embedding of $\kappa$.
Then $\kappa$ is a strong cardinal limit of strong cardinals in $M$.
\end{lemma}

\begin{lemma}\label{lem11}
Suppose $\kappa$ is $\lambda$-supercompact for some strong limit cardinal $\lambda$ of cofinality greater than $\kappa$.
Let $j:V \rightarrow M$ be a $\lambda$-supercompact embedding such that $M \vDash `` \kappa$ is not $\lambda$-supercompact$"$.
Then in $M$, there is no strong cardinal in $(\kappa,\lambda]$.
\end{lemma}

\subsection{Forcing and large cardinals} In this subsection,
we recall some well-known basic techniques
for lifting elementary embeddings. Readers can refer elsewhere for details.

For a partial order $\mathbb{P}$ and an ordinal $\kappa$,
we say $\mathbb{P}$ is $\kappa$-strategically closed if and only if in a two-person game,
in which the players construct a decreasing sequence $\langle p_{\alpha} \ | \ \alpha<\kappa \rangle$ of conditions in $\mathbb{P}$,
with Player Odd playing at odd stages and Player Even playing at even and limit stages (choosing trivial condition at stage $0$),
player Even has a strategy to ensure the game can always be continued.
$\mathbb{P}$ is ${<}\kappa$-strategically closed if and only if for any $\alpha<\kappa$,
the game is $\alpha$-strategically closed.
We say $\mathbb{P}$ is ${<}\kappa$-directed closed if and only if any directed subset
$D \subseteq \mathbb{P}$ of size less than $\kappa$ has a lower bound in $\mathbb{P}$.
Here $D$ is directed if every two elements of $D$ have a lower bound in
$\mathbb{P}$. We say $\mathbb{P}$ is ${<}\kappa$-closed if and only if any decreasing subset
$D \subseteq \mathbb{P}$ of size less than $\kappa$ has a lower bound in $\mathbb{P}$.

We use $\add(\kappa,1)=\{f \subseteq \kappa \ | \ |f|<\kappa \}$ for the Cohen forcing that adds a subset of $\kappa$.

The following is Easton’s lemma, see \cite[Lemma 15.19]{J2003} for details.
\begin{lemma}\label{clai3}
Suppose that $G \times H$ is $V$-generic for $\mathbb{P} \times \mathbb{Q}$,
where $\mathbb{P}$ is ${<}\kappa$-closed and $\mathbb{Q}$ satisfies the $\kappa$-c.c..
Then $\mathbb{P}$ is ${<}\kappa$-distributive in $V[H]$.
In other words, $\mathrm{Ord}^{<\kappa} \cap V[G][H] \subseteq V[H]$.
\end{lemma}

\begin{theorem}[\cite{Lav78}]
Suppose $\kappa$ is supercompact. Then there is a forcing $\mathbb{P}$ such that in $V^{\mathbb{P}}$, $\kappa$ is indestructible by any ${<}\kappa$-directed closed forcing. In other words, $\kappa$ is supercompact and remains supercompact after any ${<}\kappa$-directed closed forcing.
\end{theorem}
\subsection{Gitik's construction} \label{sgc}
In this subsection, let us recall some definitions and results in \cite{Git2020}.

Suppose that $T$ is a subtree of ${}^{<\theta}2$. This means that for every $t\in T$ and $\alpha<\dom(t)$, $t\restriction\alpha\in T$,
and for all $s,t\in T$, $s<_Tt$ iff $s\subseteq t$. For any $t\in T$, we denote by $\hei(t)$ the domain of $t$ and by $\Lev_{\alpha}(T)$ the $\alpha^{\text{th}}$ level of $T$, i.e., $\Lev_{\alpha}(T)=\{ t\in T \ | \ \hei(t)=\alpha \}$.
Let $\hei(T)$ denote the height of $T$.
\begin{definition}
Suppose $T \subseteq {}^{<\theta}2$ is a tree of height $\theta$, and $\delta<\theta$ is a regular cardinal.
\begin{enumerate}
  \item $T$ is a $\theta$-\emph{Suslin tree} if every maximal antichain of $T$ has size less than $\theta$.
  \item $T$ is \emph{normal} if for any $t \in T$ and any $\hei(t)<\alpha<\theta$, there is an $s >_T t$ with $\hei(s)=\alpha$.
  \item $T$ is \emph{homogeneous} if $T_{s_0}=T_{s_1}$ for every $s_0,s_1 \in T$ in the same level, where $T_s=\{t \upharpoonright (\theta \setminus |s|) \ | \ t \in T, t \geq_T s \}$ for every $s \in T$.
  \item $T$ has a $\delta$-\emph{ascent path} if there exists a function sequence $\vec{f}=\langle f_{\alpha} \ | \ \alpha<\theta, f_{\alpha}:\delta \rightarrow \mathrm{Lev}_{\alpha}(T) \rangle$, such that for every $\alpha<\beta<\theta$, the set $\{\upsilon<\delta \ | \ f_{\alpha}(\upsilon)<_T f_{\beta}(\upsilon)\}$ is co-bounded in $\delta$.
\end{enumerate}
\end{definition}

\begin{definition}[\cite{Git2020}]
Suppose $\kappa$ is a 2-Mahlo cardinal, and $\delta<\kappa$ is a regular cardinal.
Define the forcing notion $Q_{\kappa,\delta}$ as follows:
$\langle T,\vec{f} \rangle \in Q_{\kappa,\delta}$ if
\begin{enumerate}
  \item $T \subseteq {^{<\kappa}}2$ is a normal homogeneous tree of a successor height.
  \item $\vec{f}$ is a $\delta$-ascent path through $T$.
\end{enumerate}
The order on $Q_{\kappa,\delta}$ is defined by taking end extensions.

For any $Q_{\kappa,\delta}$-generic filter $G$, we denote $\vec{f}^G=\bigcup_{\langle t, \vec{f'} \rangle \in G}\vec{f'}$ by $\langle f^G_{\alpha} \mid \alpha<\kappa \rangle$, and let $T(G)$ be the $\kappa$-tree added by $G$.
\end{definition}

\begin{definition}[\cite{Git2020}]
Suppose $G$ is a $Q_{\kappa,\delta}$-generic filter over $V$.
In $V[G]$, define the forcing $F_{\kappa,\delta}$ associated with $G$,
where $g \in F_{\kappa,\delta}$ if there is a
$\xi_g<\delta$ such that
$g=f^G_{\alpha}\upharpoonright (\delta \setminus \xi_g)$ for some $\alpha<\kappa$. Set $g_0 \leq_{F_{\kappa,\delta}} g_1$ if and only if $\xi_{g_0}=\xi_{g_1}$ and for every
$\upsilon$ with $\xi_{g_0} \leq \upsilon<\delta$,
$g_0(\upsilon) \geq_{T(G)} g_1(\upsilon)$.
We also view $F_{\kappa,\delta}$ as a set of pairs $g=\langle f^G_{\alpha},\xi_g \rangle$.
\end{definition}

\begin{fact}[\cite{Git2020}]\label{factsus}
Suppose $\kappa$ is a $2$-Mahlo cardinal, and $\delta<\kappa$ is a regular cardinal.
Let $G$ be $Q_{\kappa,\delta}$-generic over $V$. Then the following holds:
\begin{enumerate}
\item\label{ite00} $Q_{\kappa,\delta}$ is ${<}\kappa$-strategically closed and ${<}\delta$-directed closed.
\item \label{ite1} $T(G)$ is a $\delta$-ascent $\kappa$-Suslin tree.
\item\label{ite01} In $V[G]$, $\langle F_{\kappa,\delta},\leq_{F_{\kappa,\delta}} \rangle$ satisfies the $\kappa$-c.c..
\item\label{it1} $\mathrm{Add}(\kappa,1)$ is forcing equivalent to the two-step iteration $Q_{\kappa,\delta}*\dot{F}_{\kappa,\delta}$.
\end{enumerate}
\end{fact}
\begin{proof}
The proofs of \eqref{ite00}, \eqref{ite1} and \eqref{ite01} may be found in \cite{Git2020}.
For the sake of completeness, we provide the proof of \eqref{it1}.

Since every ${<}\kappa$-closed poset of size $\kappa$ is forcing equivalent to $\mathrm{Add}(\kappa,1)$,
we only need to prove that $Q_{\kappa,\delta}*\dot{F}_{\kappa,\delta}$ has a ${<}\kappa$-closed dense subset of size $\kappa$.
Consider the poset $R_0$,
where $\langle t,\vec{f},g \rangle \in R_0$ iff $\langle t,\vec{f} \rangle \in Q_{\kappa,\delta}$,
and there is a $\xi_g<\delta$ such that $g=\langle f_{\mathrm{ht}(t)-1}, \xi_g \rangle$.
For every $\langle t^1,\vec{f}^1,g^1 \rangle, \langle t^2,\vec{f}^2,g^2 \rangle \in R_0$,
$\langle t^1,\vec{f}^1,g^1 \rangle \leq_{R_0} \langle t^2,\vec{f}^2,g^2 \rangle$
iff $\langle t^1,\vec{f}^1 \rangle \leq_{Q_{\kappa,\delta}} \langle t^2,\vec{f}^2 \rangle$,
$\xi_{g^1}=\xi_{g^2}$ and $g^1(\upsilon)>_{t^1} g^2(\upsilon)$ for every $\upsilon$ with $\xi_{g^1} \leq \upsilon<\delta$.
We may view $R_0$ as a subset of $Q_{\kappa,\delta}*\dot{F}_{\kappa,\delta}$.
For every $\langle t^1,\vec{f}^1, \dot{g} \rangle \in Q_{\kappa,\delta}*\dot{F}_{\kappa,\delta}$,
strengthen $\langle t^1,\vec{f}^1 \rangle$ to a condition $\langle t,\vec{f} \rangle$,
such that $\langle t,\vec{f} \rangle$ decides that $\dot{g}$ is $g=\langle f_{\alpha},\xi_g \rangle$ with $\alpha<\hei(t)$.
We may extend $\langle t,\vec{f} \rangle$ to a condition $\langle t',\vec{f}' \rangle$
so that $f'_{\mathrm{ht}(t')-1}(\upsilon) \geq_{t'} f_{\alpha}(\upsilon)$ for every $\upsilon$ with $\xi_{g} \leq \upsilon<\delta$.
Let $g'=\langle f'_{\mathrm{ht}(t')-1}, \xi_{g} \rangle$.
Then $\langle t',\vec{f}',g' \rangle$ extends $\langle t^1,\vec{f}^1, \dot{g} \rangle$.
Thus $R_0$ is dense in $Q_{\kappa,\delta}*\dot{F}_{\kappa,\delta}$.

Now we prove that $R_0$ is ${<}\kappa$-closed.
Take any limit $\gamma<\kappa$ and any $\gamma$-downward sequence $\langle \langle t^{\alpha},\vec{f}^{\alpha},g^{\alpha} \rangle \ | \ \alpha<\gamma \rangle$ of conditions in $R_0$.
We have $\xi_{g^{\alpha}}=\xi_{g^{0}}$ for every $\alpha<\gamma$,
and for every $\alpha_0<\alpha_1<\gamma$, $g^{\alpha_1}(\upsilon)>_{t^{\alpha_1}} g^{\alpha_0}(\upsilon)$ for every $\upsilon$ with $\xi_{g^0} \leq \upsilon<\delta$.
Thus $t=\bigcup_{i<\gamma}t^{i}$ has a cofinal branch.
So we may extend $t$ by adding all cofinal branches of $t$ to the level $\hei(t)$
and get a tree $t^{\gamma}$.
Let $\vec{f}=\bigcup_{\alpha<\gamma}\vec{f}^{\alpha}:=\langle f_{\alpha} \ | \ \alpha<\mathrm{ht}(t) \rangle$,
and let $f_{\mathrm{ht}(t)}:\delta \rightarrow \mathrm{Lev}_{\mathrm{ht}(t)}(t^{\gamma})$,
so that for every $\upsilon$ with $\xi_{g^0}\leq \upsilon<\delta$,
$f_{\mathrm{ht}(t)}(\upsilon)$ is the continuation of $\bigcup_{\alpha<\mathrm{ht}(t)}f_{\alpha}(\upsilon)$.
Let $\vec{f}^{\gamma}=\vec{f}\cup \{\langle \mathrm{ht}(t), f_{\mathrm{ht}(t)} \rangle\}$,
and let $g^{\gamma}=\langle f_{\mathrm{ht}(t)}, \xi_{g^0} \rangle$.
Then $\langle t^{\gamma},\vec{f}^{\gamma},g^{\gamma} \rangle$
extends $\langle t^{\alpha},\vec{f}^{\alpha},g^{\alpha}\rangle$ for every $\alpha<\gamma$.
\end{proof}
In the above fact,
if $\kappa$ is supercompact and $\delta<\kappa$ is measurable,
then $Q_{\kappa,\delta}$ may preserve the $\delta$-strong compactness of $\kappa$ after some preparation.
The idea is that after a small ultrapower map $i_U:V\rightarrow M_U$ given by a normal measure $U$ over $\delta$,
we may lift $i_U$ to $i_U:V[G] \rightarrow M_U[G^*]$ for some $G^*$ by transfer argument.
Then $i_U''\vec{f}^G$ may generate an $i_U(F_{\kappa,\delta})$-generic object associated with $i_U(T(G))$ over $M_U[G^*]$,
which may resurrect the supercompactness of $\kappa$.

It is well-known that if there is a $\kappa$-Suslin tree, then $\kappa$ is not measurable.
We give a similar result for a $\delta$-ascent $\kappa$-Suslin tree (actually, a $\delta$-ascent $\kappa$-Aronszajn tree is enough).
\begin{lemma}\label{lemma1}
Let $\kappa>\delta$ be two regular cardinal.
If $T$ is a $\delta$-ascent $\kappa$-Suslin tree, then $\kappa$ carries no $\delta^+$-complete uniform ultrafilters.
In particular, $\kappa$ is not $\delta^+$-strongly compact.
\end{lemma}
\begin{proof}
This follows from \cite[Lemmas 3.7 and 3.38]{LR22}, but we provide here a direct proof.

Suppose not, then there exists a $\delta^+$-complete uniform ultrafilter over $\kappa$.
So we have an ultrapower map $j:V \rightarrow M$ such that $\mathrm{crit}(j)>\delta$ and $\sup(j''\kappa)<j(\kappa)$.
Let $\beta=\sup(j''\kappa)$.
Let $\vec{f}=\langle f_{\alpha} \ | \ \alpha<\kappa \rangle$ be a $\delta$-ascent path through $T$,
and denote $j(\vec{f})$ by $\langle f^*_{\alpha} \ | \ \alpha<j(\kappa) \rangle$.
By elementarity and the fact that $\mathrm{crit}(j)>\delta$, it follows that
for any $\alpha_1<\alpha_2<j(\kappa)$,
the set $\{ \upsilon<\delta \ | \ f_{\alpha_1}^*(\upsilon)<_{j(T)} f^*_{\alpha_2}(\upsilon) \}$ is co-bounded in $\delta$.
So for every $\alpha<\kappa$, there is a $\theta_{\alpha}<\delta$ such that
$f^*_{\beta}(\upsilon)>_{j(T)}f^*_{j(\alpha)}(\upsilon)=j(f_{\alpha}(\upsilon))$
for every $\upsilon$ with $\theta_{\alpha} \leq \upsilon<\delta$.
Since $\mathrm{cof}(\beta)=\kappa>\delta$, it follows that
there is an unbounded subset $A \subseteq \kappa$ and a $\theta<\delta$ such that
for each $\alpha \in A$, we have $\theta=\theta_{\alpha}$.
Hence for every $\alpha_1>\alpha_2$ in $A$,
we have $j(f_{\alpha_1}(\upsilon))>_{j(T)}j(f_{\alpha_2}(\upsilon))$ for every $\upsilon$ with $\theta \leq \upsilon<\delta$,
and thus $f_{\alpha_1}(\upsilon)>_{T}f_{\alpha_2}(\upsilon)$ for every $\upsilon$ with $\theta \leq \upsilon<\delta$ by elementarity.
Therefore, $\{s \in T \mid \exists \alpha \in A (s \leq_{T} f_{\alpha}(\theta)) \}$ is a cofinal branch through $T$,
contrary to the fact that $T$ is a $\kappa$-Suslin tree.

Thus $\kappa$ carries no $\delta^+$-complete uniform ultrafilter. This means that
$\kappa$ is not $\delta^+$-strongly compact by Theorem \ref{tKU}.
\end{proof}
By Fact \ref{factsus}, if $\kappa$ is a $2$-Mahlo cardinal, then $Q_{\kappa,\delta}$ adds a $\delta$-ascent $\kappa$-Suslin tree. So $\kappa$ is not $\delta^+$-strongly compact in $V^{Q_{\kappa,\delta}}$.
\begin{corollary}\label{kill}
Suppose $\kappa$ is a $2$-Mahlo cardinal, and $\delta<\kappa$ is a regular cardinal.
Then every cardinal less than or equal to $\kappa$ is not $\delta^+$-strongly compact in $V^{Q_{\kappa,\delta}}$.
\end{corollary}
\section{the proofs for one cardinal}\label{sec3}
Gitik in \cite[Theorem 3.9 and the comment below the theorem]{Git2020} gave a new construction of a non-strongly compact $\delta$-strongly compact cardinal.
\begin{theorem}[\cite{Git2020}]
Assume $\gch$. Let $\kappa$ be a supercompact cardinal and let $\delta < \kappa$ be a measurable cardinal.
Then there is a cofinality preserving generic extension which satisfies the following:
\begin{enumerate}
    \item $\gch$,
    \item $\kappa$ is not measurable,
    \item $\kappa$ is the least $\delta$-strongly compact cardinal.
\end{enumerate}
\end{theorem}
It seems that Gitik only gave a proof for the case that there is no inaccessible cardinal above the supercompact cardinal $\kappa$, though the technique for transferring this proof to the general case is standard.
For the sake of completeness,
we first give some details of the proof of the general case in Proposition \ref{t1}. Then
we give some variants of it.
\begin{proposition}\label{t1}
Suppose $\kappa$ is a supercompact cardinal, and $\delta<\kappa$ is a measurable cardinal.
Then for every regular $\eta<\delta$,
there is a ${<}\eta$-directed closed forcing $\mathbb{P}$,
such that $\kappa$ is the least exactly $\delta$-strongly compact cardinal in $V^{\mathbb{P}}$.
In addition, if $\mathrm{GCH}$ holds in $V$, then it holds in $V^{\mathbb{P}}$.
\end{proposition}
\begin{proof}
Let $A=\{ \alpha<\kappa \mid \alpha>\delta$ is a strong cardinal limit of strong cardinals$\}$,
and let $B=\{\alpha<\kappa \mid \alpha>\delta$ is the least inaccessible limit of strong cardinals above some ordinal$\}$.
Then $A$ and $B$ are nonempty by Lemma \ref{lem10},
and $A\cap B=\emptyset$.

Let $\mathbb{P}_{\kappa+1}=\langle \langle \mathbb{P}_{\alpha},\dot{\mathbb{R}}_{\alpha} \rangle \mid \alpha \leq \kappa \rangle$
be the Easton support iteration of length $\kappa+1$,
where $\mathbb{\dot{R}}_{\alpha}$ is
\begin{itemize}
  \item a $\mathbb{P}_{\alpha}$-name for $\mathrm{Add}(\alpha,1)_{V^{\mathbb{P}_{\alpha}}}$ if $\alpha \in A$,
  \item a $\mathbb{P}_{\alpha}$-name for $(Q_{\alpha,\eta})_{V^{\mathbb{P}_{\alpha}}}$ if $\alpha \in B$,
  \item a $\mathbb{P}_{\alpha}$-name for $(Q_{\kappa,\delta})_{V^{\mathbb{P}_{\alpha}}}$ if $\alpha=\kappa$,
  \item the trivial forcing, otherwise.
\end{itemize}
Let $\mathbb{P}=\mathbb{P}_{\kappa+1}$.
Let $G_{\kappa}$ be a $\mathbb{P}_{\kappa}$-generic filter over $V$,
let $g$ be an $\mathbb{R}_{\kappa}$-generic filter over $V[G_{\kappa}]$, and let $G=G_{\kappa}*g \subseteq \mathbb{P}$.
For every $\alpha<\kappa$, let $G_{\alpha}=\{p \upharpoonright \alpha \mid p \in G_{\kappa}\}$,
and let $\dot{\mathbb{P}}_{\alpha,\kappa+1}$ name the canonical iteration of length $\kappa+1-\alpha$ such that $\mathbb{P}$
is forcing equivalent to $\mathbb{P}_{\alpha}*\dot{\mathbb{P}}_{\alpha,\kappa+1}$.
Then $\mathbb{P}$ is ${<}\eta$-directed closed,
because for every $\alpha<\kappa$, $\mathbb{P}_{\alpha}$ forces that $\dot{\mathbb{R}}_{\alpha}$ is ${<}\eta$-directed closed.
In addition, if $\mathrm{GCH}$ holds in $V$, then it holds in $V^{\mathbb{P}}$.

Take any $\alpha \in B$. In $V[G_{\alpha}]$,
the forcing $\mathbb{R}_{\alpha}=Q_{\alpha,\eta}$ adds an $\eta$-ascent $\alpha$-Suslin tree.
So $\alpha$ carries no $\eta^+$-complete uniform ultrafilters in $V[G_{\alpha+1}]$ by Lemma \ref{lemma1}.
Meanwhile, $\mathbb{P}_{\alpha+1,\kappa+1}$ is $(2^\alpha)^+$-strategically closed in $V[G_{\alpha+1}]$.
So any ultrafilter over $\alpha$ in $V[G]$ is actually in $V[G_{\alpha+1}]$.
Hence, there are no $\eta^+$-complete uniform ultrafilters over $\alpha$ in $V[G]$.
This implies that there is no $\eta^+$-strongly compact cardinal below $\alpha$ in $V[G]$ by Theorem \ref{tKU}.
Note that $B$ is unbounded in $\kappa$,
it follows that there is no $\eta^+$-strongly compact cardinal below $\kappa$.

Now we are ready to prove that $\kappa$ is exactly $\delta$-strongly compact.
By Corollary \ref{kill}, the cardinal $\kappa$ is not $\delta^+$-strongly compact.

Let $W$ be a normal measure over $\delta$, and let $i_W:V \rightarrow M_W$ be the corresponding ultrapower map.
Note that $\mathbb{P}$ is ${<}\delta^+$-strategically closed,
we may lift $i_W$ to $i_W:V[G] \rightarrow M_W[G^*]$ by transfer argument.
Here $G^*:=G_{\kappa}^**g^* \subseteq i_W(\mathbb{P})$ is the filter generated by $i_W''G$.

Next, we show that in $V[G]$,
there is an $i_W(F_{\kappa,\delta})$-generic object over $M_W[G^*]$.
In $M_W[G^*]$, the sequence $\vec{f}^{g^*}=\langle f^{g^*}_{\alpha} \mid \alpha<\kappa \rangle:=i_W(\vec{f}^g)$ is an $i_W(\delta)$-ascent path through
the $\kappa$-Suslin tree $T(g^*)$.
And for every $\upsilon$ with $\delta \leq \upsilon<i_W(\delta)$,
$\{f^{g^*}_{i_W(\alpha)}(\upsilon) \mid \alpha<\kappa \}$ generates a cofinal branch $\{ t \in T(g^*) \mid \exists \alpha<\kappa (t<_{T(g^*)} f^{g^*}_{i_W(\alpha)}(\upsilon))\}$ through $T(g^*)$.

We claim that the set $\{ \langle f^{g^*}_{i_W(\alpha)},\delta \rangle \mid \alpha<\kappa \}$ generates an $i_W(F_{\kappa,\delta})$-generic filter $F^*$ over $M_W[G^*]$.
In other words, let
\[
F^*=\{\langle f^{g^*}_{\gamma},\delta \rangle \mid \gamma<\kappa, \exists \gamma<\alpha<\kappa (\forall \delta \leq \upsilon <i_W(\delta)(f^{g^*}_{\gamma}(\upsilon)<_{T(g^*)}f^{g^*}_{i_W(\alpha)}(\upsilon)))\}.
\]
Then $F^*$ is an $M_W[G^*]$-generic set for the forcing
$i_W(F_{\kappa,\delta})$ associated with the tree $T(g^*)$.
We only need to prove that for every maximal antichain $A$ of $i_W(F_{\kappa,\delta})$,
there is an element of $F^*$ extends some element of $A$.
In $M_W[G^*]$,
note that $i_W(F_{\kappa,\delta})$ has $\kappa$-c.c., it follows that
for every maximal antichain $A$ of $i_W(F_{\kappa,\delta})$,
we may find an ordinal $\alpha<\kappa$
such that each function of $A$ acts at a level of $T(g^*)$ below $\alpha$.
Now take an inaccessible cardinal $\alpha<\beta<\kappa$ in $V[G]$.
Then $\langle f^{g^*}_{\beta},\delta \rangle \in F^*$
extends some element of $A$.
Thus $M_W[G^*,F^*]$ is an $i_W(\mathbb{P}_{\kappa} * \dot{\mathrm{Add}}(\kappa,1))$-generic extension of $M_W$.

Now work in $M_W$.
Let $\lambda > \kappa$ be an arbitrary singular strong limit cardinal of cofinality greater than $\kappa$.
Let $U$ be a normal measure over $\mathcal{P}_{\kappa}(\lambda)$,
so that the embedding $j:=j^{M_W}_U:M_W \rightarrow N \cong \mathrm{Ult}(M_W,U)$ satisfies that $N \vDash ``\kappa$ is not $\lambda$-supercompact$"$. Let $\pi=j\circ i_W$.

Work in $M_W[G^*,F^*]$, which is an $i_W(\mathbb{P}_{\kappa} * \dot{\mathrm{Add}}(\kappa,1))$-generic extension of $M_W$.
Note that $\kappa \in \pi(A)$ by Lemma \ref{lem10},
and in $N$, there is no strong cardinal in $(\kappa,\lambda]$ by Lemma \ref{lem11},
a standard argument shows that we may lift $j$ and obtain an embedding $j:M_W[G^*]\rightarrow N[G^*,F^*,H]$. Here $H$ is a $\pi(\mathbb{P})/ (G^**F^*)$-generic filter constructed in $M_W[G^*,F^*]$(see \cite{C2010} or \cite[Lemma 4]{AC2001} for the detail).

Let $D=j''\lambda \in N[G^*,F^*,H]$, then $\pi''\lambda=D$ and $N[G^*,F^*,H] \vDash |D|<\pi(\kappa)$.
Hence in $V[G]$, $\pi$ witnesses that $\kappa$ is $(\delta,\lambda)$-strongly compact.
As $\lambda$ is an arbitrarily strong limit cardinal of cofinality greater than $\kappa$,
it follows that $\kappa$ is $\delta$-strongly compact.

This completes the proof of Proposition \ref{t1}.
\end{proof}

\begin{remark}\label{rmk2}
Take any inaccessible $\theta$ with $\delta<\theta<\kappa$. For any ${<}\theta$-strategically closed forcing $\mathbb{Q}_0 \in V_{\kappa}$
and any forcing $\mathbb{Q}_1 \in V_{\theta}$ such that $i_W$ can be lifted to some embedding with domain $V^{\mathbb{Q}_1}$,
$\kappa$ is the least exactly $\delta$-strongly compact cardinal in $V^{\mathbb{P}\times \mathbb{Q}_0 \times \mathbb{Q}_1}$.
\end{remark}
It is easy to see that $\kappa$ is the least $\eta^+$-strongly compact cardinal in the proposition above.

If we add a $\delta$-ascent Suslin tree at the least $2$-Mahlo cardinal above $\kappa$ instead of $\kappa$, then $\kappa$ may be measurable.
Thus the least $\delta$-strongly compact cardinal may be non-strongly compact and measurable.
We may compare this with Menas' theorem that every measurable almost strongly compact cardinal must be strongly compact.
\begin{proposition}\label{t2}
Suppose $\kappa$ is a supercompact cardinal, $\kappa'$ is the least $2$-Mahlo cardinal above $\kappa$,
and $\delta<\kappa$ is a measurable cardinal.
Then for any $\eta<\delta$, there is a ${<}\eta$-directed closed forcing $\mathbb{P}$,
such that in $V^{\mathbb{P}}$, $\kappa$ is the least exactly $\delta$-strongly compact cardinal, and $\kappa$ is measurable.
\end{proposition}
\begin{proof}
Without loss of generality,
we may assume that $2^{\kappa}=\kappa^+$.
Take the Easton support forcing $\mathbb{P}=\mathbb{P}_{\kappa+1}$ in Proposition \ref{t1},
but with $A=\{ \alpha<\kappa \mid \alpha>\delta$ is the first $2$-Mahlo cardinal above some strong cardinal limit of strong cardinals$\}$,
and $\dot{\mathbb{R}}_{\kappa}$ names $(Q_{\kappa',\delta})_{V^{\mathbb{P}_{\kappa}}}$ instead of $(Q_{\kappa,\delta})_{V^{\mathbb{P}_{\kappa}}}$.
Then in $V^{\mathbb{P}}$, $\kappa$ is the least exactly $\delta$-strongly compact by the argument in Proposition \ref{t1}
and measurable by a standard lifting argument (see \cite[Theorem 11.1]{C2010}).
\end{proof}
In the above proposition,
note that $A \cup B$ is sparse and every element in $A \cup B$ is not measurable,
it is easy to see that if $\gch$ holds, then every measurable cardinal is preserved
by a standard lifting argument (see \cite[Theorem 11.1]{C2010}).

\begin{corollary}\label{t4}
Suppose $\kappa$ is supercompact and indestructible under any ${<}\kappa$-directed forcing, and $\delta<\kappa$ is a measurable cardinal.
Then $\kappa$ is exactly $\delta$-strongly compact in $V^{Q_{\kappa,\delta}}$. In addition,
if $\kappa'$ is the least $2$-Mahlo cardinal above $\kappa$,
then $\kappa$ is exactly $\delta$-strongly compact and measurable in $V^{Q_{\kappa',\delta}}$.
\end{corollary}

\section{main theorems}\label{sec4}
In this section, by the proof of \cite[Theorem 2]{Apter(1996)}, for a given class $\mathcal{K}$ of supercompact cardinals,
we may assume that $V \vDash ``\mathrm{ZFC}+\mathrm{GCH}+ \mathcal{K}$ is the class of supercompact cardinals
$+$
Every supercompact cardinal $\kappa$ is Laver indestructible
under any ${<}\kappa$-directed closed forcing not destroying $\mathrm{GCH}$ $+$
The strongly compact cardinals and supercompact cardinals coincide precisely,
except possibly at measurable limit points$"$.

\begin{theorem}\label{thm1}
Suppose that $\mathcal{A}$ is a subclass of the class $\mathcal{K}$ of the supercompact cardinals containing none of its limit points,
and $\delta_{\kappa}$ is a measurable cardinal with $\sup(\mathcal{A}\cap \kappa)< \delta_{\kappa}<\kappa$ for every $\kappa \in \mathcal{A}$.
Then there exists a forcing extension, in which
$\kappa$ is the least exactly $\delta_{\kappa}$-strongly compact cardinal for every $\kappa \in \mathcal{A}$.
In addition, no new strongly compact cardinals are created and $\gch$ holds.
\end{theorem}
\begin{proof}
For each $\kappa \in \mathcal{A}$,
let $\eta_{\kappa}=(2^{\sup(\mathcal{A}\cap \kappa)})^+$,
with $\eta_{\kappa}=\omega_1$ for $\kappa$ the least element of $\mathcal{A}$.
Let $\mathbb{P}_{\kappa}$ be a forcing that forces $\kappa$
to be the least exactly $\delta_{\kappa}$-strongly compact cardinal
by first taking $\eta=\eta_{\kappa}$ and $\delta=\delta_{\kappa}$ and
then using any of ${<}\eta$-directed closed forcing given by Proposition \ref{t1}.
Let $\mathbb{P}$ be the Easton support product $\Pi_{\kappa \in \mathcal{A}}\mathbb{P}_{\kappa}$.
Note that the forcing is a product, and the field of $\mathbb{P}_{\kappa}$ lies in $(\delta_{\kappa},\kappa]$,
which contains no elements of $\mathcal{A}$,
so the fields of the forcings $\mathbb{P}_{\kappa}$ occur in different blocks.
Though $\mathbb{P}$ is a class forcing,
the standard Easton argument shows $V^{\mathbb{P}}\vDash \mathrm{ZFC}+\mathrm{GCH}$.

\begin{lemma}
If $\kappa \in \mathcal{A}$, then $V^{\mathbb{P}} \vDash ``\kappa$ is the least exactly $\delta_{\kappa}$-strongly compact cardinal$"$.
\end{lemma}
\begin{proof}
We may factor the forcing  in $V$ as $\mathbb{P}=\mathbb{Q}^{\kappa} \times \mathbb{P}_{\kappa} \times \mathbb{Q}_{<\kappa}$,
where $\mathbb{Q}^{\kappa}=\Pi_{\alpha>\kappa}\mathbb{P}_{\alpha}$ and $\mathbb{Q}_{<\kappa}=\Pi_{\alpha<\kappa}\mathbb{P}_{\alpha}$.

By indestructibility and the fact that $\mathbb{Q}^{\kappa}$ is a ${<}\kappa$-directed closed forcing not destroying $\gch$,
it follows that $\kappa$ is supercompact in $V^{\mathbb{Q}^{\kappa}}$.
Note that the strong cardinals below $\kappa$ are the same in $V^{\mathbb{Q}^{\kappa}}$ and $V$, because $\mathbb{Q}^{\kappa}$ is ${<}\kappa$-directed closed and $\kappa$ is supercompact in $V^{\mathbb{Q}^{\kappa}}$.
Therefore, the forcing $\mathbb{P}_{\kappa}$ satisfies the same definition in either $V$ or $V^{\mathbb{Q}^{\kappa}}$. So $\kappa$ is the least $\delta_{\kappa}$-strongly compact cardinal in $V^{\mathbb{Q}^{\kappa}\times \mathbb{P}_{\kappa}}$ by Proposition \ref{t1}.

Note that $\delta_{\kappa}>\sup(\mathcal{A}\cap \kappa)$ is a measurable cardinal, it follows that $|\mathbb{Q}_{<\kappa}|<\delta_{\kappa}$. So $V^{\mathbb{P}}=V^{\mathbb{Q}^{\kappa} \times \mathbb{P}_{\kappa} \times \mathbb{Q}_{<\kappa}}\vDash ``\kappa$ is the least exactly $\delta_{\kappa}$-strongly compact cardinal$"$.
\end{proof}
Now we only need to prove that no new strongly compact cardinal are created.
In other words, if $\theta$ is strongly compact in $V^{\mathbb{P}}$,
then $\theta \in \mathcal{K}$ or it is a measurable limit point of $\mathcal{K}$ in $V$.
Suppose not.
If $\sup(\mathcal{A})<\theta$, then $|\mathbb{P}|<\theta$.
This implies that $\theta$ is also strongly compact in $V$, contrary to our assumption.
So we may assume that $\theta\leq \sup(\mathcal{A})$.
Then there are two cases to consider:

$\br$ $\theta$ is not a limit point of $\mathcal{A}$.
Let $\kappa=\min(\mathcal{A} \setminus (\theta+1))$.
Then $\theta \in (\eta_{\kappa},\kappa)$.
We may factor $\mathbb{P}$ as $\mathbb{Q}^{\kappa} \times \mathbb{P}_{\kappa} \times \mathbb{Q}_{<\kappa}$.
Then in $V^{\mathbb{P}_{\kappa}}$, $\theta$ is strongly compact,
which implies that $\theta$ is $\eta_{\kappa}^+$-strongly compact.
However, $\kappa$ is the least $\eta_{\kappa}^+$-strongly compact cardinal in $V^{\mathbb{P}_{\kappa}}$ by Proposition \ref{t1}, a contradiction.

$\br$ $\theta$ is a limit of $\mathcal{A}$. For the sake of completeness,
we give an outline of the proof here (see \cite[Lemma 8]{Apter(1996)} for details).

Note that $\theta$ is measurable in $V^{\mathbb{P}}$, it is easy to see that $\theta$ is measurable in $V^{\mathbb{Q}_{< \theta}}:=V^{\Pi_{\alpha<\theta}\mathbb{P}_{\alpha}}$. In addition, $\mathbb{Q}_{< \theta}$ has $\theta$-c.c. in $V^{\mathbb{Q}_{< \theta}}$.

Work in $V^{\mathbb{Q}_{< \theta}}$. Let $\mu$ be a normal measure over $\theta$.
We can prove that there exists a $q \in \mathbb{Q}_{<\theta}$, such that for every $X \subseteq \theta$ in $V$,
we have $q$ decides $X \in \dot{\mu}$, i.e.,
$q \Vdash ``X \in \dot{\mu}"$ or $q \Vdash ``X \notin \dot{\mu}"$.
Otherwise, we can construct a $\theta$-tree
$T=\{\langle r_s,X_s \rangle\mid s \in \ ^{<\theta}2, r_s \in \mathbb{Q}_{<\theta}, X_s \subseteq \theta\}$ such that
\begin{enumerate}
  \item $r_s \Vdash ``X_s \in \dot{\mu}"$,
  \item if $s' \subseteq s$ then $r_{s}\leq_{\mathbb{Q}_{< \theta}} r_{s'}$,
  \item $X_s=X_{s^{\frown}\langle 0 \rangle}\cup X_{s^{\frown}\langle 1 \rangle}$ and $X_{s^{\frown}\langle 0 \rangle}\cap X_{s^{\frown}\langle 1 \rangle}=\emptyset$.
  \item If $\dom(s)$ is a limit ordinal, then $X_s=\cap_{s' \subsetneq s}X_{s'}$.
\end{enumerate}
Note that $\theta$ is weakly compact in $V^{\mathbb{Q}_{<\theta}}$,
it follows that $T$ has a cofinal branch $\langle \langle r_s,X_s \rangle \mid s\in \ ^{<\theta}2, \ s \subseteq f \rangle$ for some function $f:\theta \rightarrow 2$.
Then $\{r_{s^{\frown}\langle i \rangle} \mid  i=0 \vee i=1, s\in \ ^{<\theta}2, \ s \subseteq f, s^{\frown}\langle i \rangle \nsubseteq f\}$
is an antichain of $\mathbb{Q}_{<\theta}$ of size $\theta$.
However, $\mathbb{Q}_{<\theta}$ has $\theta$-c.c., a contradiction.
Hence there exists a $q \in \mathbb{Q}_{<\theta}$, such that
$q$ decides $X \in \dot{\mu}$ for every $X \subseteq \theta$ in $V$, i.e., $\theta$ is measurable in $V$.

This completes the proof of Theorem \ref{thm1}.
\end{proof}

In the above proof, if we let $\mathbb{P}_{\kappa}$ be the poset in Proposition \ref{t2} instead of the poset in Proposition \ref{t1},
then $\kappa$ remains measurable in $V^{\mathbb{P}}$.

If there is no measurable limit point of $\mathcal{K}$ and $\mathcal{A}=\mathcal{K}$,
then there is no strongly compact cardinal in $V^{\mathbb{P}}$.
\begin{corollary}
Suppose $\mathcal{K}$ is the class of supercompact cardinals with no measurable limit points, and $\delta_{\kappa}$ is measurable with $\sup(\mathcal{K}\cap \kappa)<\delta_{\kappa}<\kappa$ for any $\kappa \in \mathcal{K}$. Then there exists a forcing extension, in which $\kappa$ is the least exactly $\delta_{\kappa}$-strongly compact cardinal for any $\kappa \in \mathcal{K}$, and there is no strongly compact cardinal.
\end{corollary}
In the corollary above, note that in the forcing extension, every limit point of $\mathcal{K}$ is almost strongly compact,
we may have a model in which the least almost strongly compact cardinal is not strongly compact.
\begin{theorem}\label{coric}
Suppose that $\langle \kappa_n \mid n<\omega \rangle$ is an increasing sequence of supercompact cardinals. Let $\kappa=\lim_{n<\omega}\kappa_n$. Then there is a forcing extension, in which $\kappa$ is the least almost strongly compact cardinal.
\end{theorem}
\begin{proof}
Let $\kappa_{-1}=\omega_1$ for simplicity. Let $\delta_n$ be the least measurable cardinal greater than $\kappa_{n-1}$ for every $n<\omega$.
Then by Theorem \ref{thm1}, there is a forcing extension, say $V^{\mathbb{P}}$,
in which $\kappa_n$ is the least exactly $\delta_n$-strongly compact cardinal for every $n<\omega$.
Thus $\kappa$ is $\delta_n$-strongly compact cardinal for every $n<\omega$.
Note also that $\kappa=\lim_{n<\omega}\delta_n$,
we have $\kappa$ is almost strongly compact.

Take any $n<\omega$.
Then in $V^{\mathbb{P}}$, there are unboundedly many $\alpha$ below $\kappa_n$
such that there is no $(2^{\kappa_{n-1}})^{++}$-complete uniform ultrafilter over $\alpha$.
So there is no almost strongly compact cardinal in $(\kappa_{n-1},\kappa_n)$.
Since $n<\omega$ is arbitrary, it follows that there is no almost strongly compact cardinal below $\kappa$.
Thus $\kappa$ is the least almost strongly compact cardinal.
\end{proof}
This answers Question \ref{q2} in the negative.

If there exists a proper class of supercompact cardinals, we may get a model in which there exists a proper class of almost strongly compact cardinals, but there are no strongly compact cardinals.
\begin{corollary}
Suppose there is a proper class of supercompact cardinals with no measurable limit points. Then there exists a forcing extension, in which there is a proper class of almost strongly compact cardinals, but there are no strongly compact cardinals.
\end{corollary}
Next, we deal with another case.
\begin{theorem}\label{thm2}
Suppose $\mathcal{K}$ is a set of supercompact cardinals and has order type less than or equal to  $\min(\mathcal{K})+1$, and $\langle \delta_{\kappa} \mid \kappa \in \mathcal{K} \rangle$ is an increasing sequence of measurable cardinals with $\sup_{\kappa\in \mathcal{K}}\delta_{\kappa} \leq \min(\mathcal{K})$. Then there exists a forcing extension, in which $\kappa$ is exactly $\delta_{\kappa}$-strongly compact for any $\kappa \in \mathcal{K}$.
Moreover, if $\kappa=\min(\mathcal{K})$, then $\kappa$ is the least exactly $\delta_{\kappa}$-strongly compact cardinal. In addition, there is no strongly compact cardinal and $\gch$ holds.
\end{theorem}
\begin{proof}
For every $\kappa \in \mathcal{K}$, let $\kappa'$ be the least $2$-Mahlo cardinal above $\kappa$ if $\kappa<\sup(\mathcal{K})$, and $\kappa$, otherwise. Let $\kappa_0=\min(\mathcal{K})$. Let $\mathbb{P}_{\kappa}$ be a forcing given by Proposition $\ref{t2}$ with $\eta=\omega_1$ if $\kappa=\kappa_0$, and $Q_{\kappa',\delta_{\kappa}}$ if $\kappa\in \mathcal{K}\setminus (\kappa_0+1)$. Let $\mathbb{P}$ be the Easton product forcing $\Pi_{\kappa \in \mathcal{K}}\mathbb{P}_{\kappa}$. Then $V^{\mathbb{P}}\vDash \zfc +\gch$.
Note that $\gch$ holds in $V$,
it follows that $\delta_{\kappa}$ remains measurable in $V^{\mathbb{P}}$
by the comment below Proposition $\ref{t2}$.

For every $\kappa \in \mathcal{K}$, we may factor $\mathbb{P}$ in $V$ as $\mathbb{Q}^{\kappa} \times \mathbb{P}_{\kappa} \times \mathbb{Q}_{<\kappa}$, where $\mathbb{Q}^{\kappa}=\Pi_{\alpha>\kappa}\mathbb{P}_{\alpha}$ and $\mathbb{Q}_{<\kappa}=\Pi_{\alpha<\kappa}\mathbb{P}_{\alpha}$.
Take any $\kappa \in \mathcal{K}$, and let $\lambda_{\kappa}=\min(\mathcal{K}\setminus (\kappa+1))$
(if $\kappa=\max(\mathcal{K})$, then let $\lambda_{\kappa}=\infty)$.
Then it is easy to see that $\mathbb{Q}^{\kappa}$ is ${<}\lambda_{\kappa}$-strategically closed.
Note also that $\kappa$ is supercompact and indestructible by any ${<}\kappa$-directed closed forcing not destroying $\gch$, we have
$V^{\mathbb{Q}^{\kappa}}\vDash ``\kappa$ is  ${<}\lambda_{\kappa}$-supercompact and indestructible by any ${<}\lambda_{\kappa}$-directed closed forcing not destroying $\gch"$.

We may factor $\mathbb{Q}_{<\kappa}$ as the product of $\mathbb{P}_{\kappa_0}$ and a ${<}\lambda_{\kappa_0}$-strategically closed forcing in $V_{\kappa}$.
Note also that if $i_W: V^{\mathbb{Q}^{\kappa}} \rightarrow M$ is an ultrapower map given by a normal measure $W$ over $\delta_{\kappa}$ such that $\kappa$ is not measurable in $M$,
then we may lift $i_W$ to an embedding with domain $V^{\mathbb{Q}^{\kappa}\times \mathbb{P}_{\kappa_0}}$.
Thus it is not hard to see that
$\kappa$ is $(\delta_{\kappa},\lambda_{\kappa})$-strongly compact in $V^{\mathbb{P}}$ by the argument in Proposition \ref{t1} and Remark \ref{rmk2}.

Since $\kappa \in \mathcal{K}$ is arbitrary, it follows that $\kappa$ is exactly $\delta_{\kappa}$-strongly compact in $V^{\mathbb{P}}$ by Theorem \ref{tKU}.
Moreover, $\kappa_0$ is the least $\delta_{\kappa_0}$-strongly compact cardinal in $V^{\mathbb{P}}$.
In addition, there is no strongly compact cardinal in $(\delta_{\kappa},\kappa]$ for every $\kappa \in \mathcal{K}$.
This means that there is no strongly compact cardinal $\leq \sup(\mathcal{K})$.
So there is no strongly compact cardinal in $V^{\mathbb{P}}$ by our assumption at the beginning of this section.
\end{proof}
In the above theorem, $\kappa \in \mathcal{K}$ may not be the least $\delta_{\kappa}$-strongly compact cardinal if $\kappa \neq \kappa_0$.
We used the forcing given by Proposition \ref{t2} instead of the forcing given by Proposition \ref{t1} for $\kappa=\kappa_0$, because
we need to preserve the measurability of $\kappa_0$ if $\kappa_0=\delta_{\max(\mathcal{K})}$.

Next, we combine Theorem \ref{thm1} and Theorem \ref{thm2}.
\begin{theorem}\label{thm7}
Suppose that $\mathcal{A}$ is a subclass of the class $\mathcal{K}$ of the supercompact cardinals containing none of its limit points,
and $\langle \delta_{\kappa} \mid \kappa \in \mathcal{K}\rangle$ is an increasing sequence of measurable cardinals such that $\delta_{\kappa}<\kappa$ 
for any $\kappa\in \mathcal{A}$.
Then in some generic extension,
$\kappa$ is exactly $\delta_{\kappa}$-strongly compact for any $\kappa \in \mathcal{A}$.
Moreover, $\kappa$ is the least exactly $\delta_{\kappa}$-strongly compact cardinal
if $\delta_{\kappa}>\sup(\mathcal{A} \cap \kappa)$.
\end{theorem}
\begin{proof}
For every $\kappa \in \mathcal{A}$, let $\eta_{\kappa}=(2^{\sup(\mathcal{A}\cap \kappa)})^+$,
and let $\kappa'$ be the least $2$-Mahlo cardinal above $\kappa$ if $\kappa<\sup(\mathcal{K})$, and $\kappa$, otherwise.
Let $\mathbb{P}_{\kappa}$ be
\begin{itemize}
  \item the forcing $Q_{\kappa',\delta_{\kappa}}$ if $\delta_{\kappa}{\leq}\sup(\mathcal{A}\cap \kappa)$,
  \item the forcing given by Proposition \ref{t1} with $\eta=\eta_{\kappa}$ and $\delta=\delta_{\kappa}$ if $\kappa=\sup(\mathcal{A})$,
  \item the forcing given by Proposition \ref{t2} with $\eta=\eta_{\kappa}$ and $\delta=\delta_{\kappa}$, otherwise.
\end{itemize}
Let $\mathbb{P}$ be the Easton product forcing $\Pi_{\kappa \in \mathcal{A}}\mathbb{P}_{\kappa}$.
Then for every measurable cardinal $\delta$, if $\delta \neq \sup(\mathcal{A}\cap \kappa)$,
then $\delta$ remains measurable in $V^{\mathbb{P}}$.

Take any $\kappa \in \mathcal{A}$.
Again we may factor $\mathbb{P}$ as
$\mathbb{Q}^{\kappa} \times \mathbb{P}_{\kappa} \times \mathbb{Q}_{<\kappa}$.
Let $\lambda_{\kappa}=\min(\mathcal{A}\setminus (\kappa+1))$.
Then there are two cases to consider:

$\br$ $\delta_{\kappa}{>}\sup(\mathcal{A}\cap \kappa)$.
If $\delta_{\lambda_{\kappa}}> \kappa$,
then $\mathbb{Q}^{\kappa}$ is ${<}\kappa$-directed closed, and $\mathbb{Q}_{\kappa}$ has size less than $\delta_{\kappa}$.
So $\kappa$ is the least exactly $\delta_{\kappa}$-strongly compact cardinal in $V^{\mathbb{P}}$ by the argument in Theorem \ref{thm1}.

If $\delta_{\lambda_{\kappa}} \leq \kappa$,
then $\kappa$ is ${<}\lambda_{\kappa}$-supercompact in $V^{\mathbb{Q}^{\kappa}}$.
Note that $\mathbb{Q}_{\kappa}$ has size less than $\delta_{\kappa}$,
it follows that $\kappa$ is $(\delta_{\kappa},\lambda_{\kappa})$-strongly compact in $V^{\mathbb{P}}$.

$\br$ $\delta_{\kappa}{\leq}\sup(\mathcal{A}\cap \kappa)$.
If $\delta_{\kappa}<\sup(\mathcal{A}\cap \kappa)$, then $\sup(\mathcal{A}\cap \delta_{\kappa})<\delta_{\kappa}$.
We may factor $\mathbb{Q}_{<\kappa}$ as the product of $\Pi_{\theta \in \mathcal{A}, \delta_{\kappa}\leq \theta<\kappa}\mathbb{P}_{\theta}$ and $\Pi_{\theta \in \mathcal{A}\cap \delta_{\kappa}}\mathbb{P}_{\theta}$.
Then $\kappa$ is exactly $(\delta_{\kappa},\lambda_{\kappa})$-strongly compact
in $V^{\mathbb{Q}^{\kappa} \times \mathbb{P}_{\kappa} \times \Pi_{\theta \in \mathcal{A}, \delta_{\kappa}\leq \theta<\kappa}\mathbb{P}_{\theta}}$ by Theorem \ref{thm2}.
Note also that $\Pi_{\theta \in \mathcal{A}\cap \delta_{\kappa}}\mathbb{P}_{\theta}$ has size less than $\delta_{\kappa}$,
it follows that $\kappa$ is exactly $(\delta_{\kappa},\lambda_{\kappa})$-strongly compact in $V^{\mathbb{P}}$.

If $\delta_{\kappa}=\sup(\mathcal{A}\cap \kappa)$, then $\delta_{\kappa}$ is a measurable limit point of $\mathcal{A}$ in $V$ and $\mathcal{A}\cap [\delta_{\kappa},\kappa)=\emptyset$.
Then $\kappa$ is exactly $(\delta_{\kappa},\lambda_{\kappa})$-strongly compact by Remark \ref{rmk2}.

Thus in $V^{\mathbb{P}}$, $\kappa$ is exactly $\delta_{\kappa}$-strongly compact for any $\kappa \in \mathcal{A}$.
Moreover, $\kappa$ is the least exactly $\delta_{\kappa}$-strongly compact cardinal
if $\delta_{\kappa}>\sup(\mathcal{A} \cap \kappa)$ in $V^{\mathbb{P}}$.

This completes the proof of Theorem \ref{thm7}.
\end{proof}
This answers Question \ref{QBM} affirmatively.
\section{Open Problems}
By the results of Goldberg in \cite{G2020}, it is not hard to see the following theorem holds:
\begin{theorem}[\cite{G2020}]
Suppose $\kappa$ is an almost strongly compact cardinal. Then exactly one of the following holds:
\begin{enumerate}
  \item\label{pro1} $\kappa$ has cofinality $\omega$, and $\kappa$ is not a limit of almost strongly compact cardinals.
  \item\label{pro2} $\kappa$ is the least $\omega_1$-strongly compact cardinal, but $\kappa$ is not strongly compact.
  \item\label{pro3} $\kappa$ is a strongly compact cardinal.
  \item\label{pro4} $\kappa$ is the successor of a strongly compact cardinal.
  \item\label{pro5} $\kappa$ is a non-strongly compact limit of almost strongly compact cardinals.
\end{enumerate}
\end{theorem}
According to Theorem \ref{coric}, we proved that consistently \eqref{pro1} is possible under suitable large cardinal assumptions.
\eqref{pro3}, \eqref{pro4} and \eqref{pro5} are trivial.
However, we don't know whether \eqref{pro2} is possible or not.
\begin{question}
If the least almost strongly compact cardinal is the least $\omega_1$-strongly compact cardinal, is it necessarily strongly compact?
\end{question}
Under the assumption of $\sch$, Goldberg in \cite[Theorem 5.7]{G2020} gave an affirmative answer to this question.
So to get a negative consistency result, one has to violate $\mathrm{SCH}$ at unboundedly many cardinals below $\kappa$.
\bibliography{ref}
\end{document}